\documentclass[12pt]{elsarticle}
\usepackage{amsmath}
\usepackage{amsfonts}
\usepackage{amssymb}
\usepackage{amsthm}
\usepackage{color}
\usepackage{graphicx}
\usepackage{epsfig,mathrsfs}
\usepackage[bf,SL,BF]{subfigure}
\usepackage{fancyhdr}

\usepackage{caption}
\usepackage{wrapfig}
\usepackage{cases}
\usepackage{algorithm}
\usepackage{algpseudocode}
\usepackage{multirow}
\newtheorem{thm}{Theorem}[section]

\newtheorem{lem}[thm]{Lemma}

\theoremstyle{definition}
\newtheorem{defn}[thm]{Definition}
\theoremstyle{remark}
\newtheorem{rem}[thm]{Remark}
\numberwithin{equation}{section}
\newcommand{\fai}{\pmb{\phi}}
\newcommand{\pai}{\pmb{\pi}}
\newcommand{\q}{\textbf{\emph{q}}}
\newcommand{\p}{\textbf{\emph{p}}}
\newcommand{\x}{\textbf{\emph{x}}}
\newcommand{\bm}[1]{\emph{\textbf{#1}}}
\newcommand{\n}{\textbf{n}}
\newcommand{\R}{\mathcal{R}}
\newcommand{\RL}{Riemann-Liouville }
\newcommand{\cinf}[1]{C_0^{\infty}(#1)}

\newcommand{\norm}[1]{\left\Vert#1\right\Vert}
\newcommand{\flux}[1]{\widehat{{#1}}}
\makeatletter

\newcommand{\Rmnum}[1]{\mathcal{\expandafter\@slowromancap\romannumeral #1@}}
\makeatother

\begin{document}

\pagenumbering{arabic}

\begin{frontmatter}
\title{{\bf  Nodal discontinuous Galerkin methods for fractional diffusion equations on 2D domain with triangular meshes}}
\author{Liangliang Qiu$^1$, Weihua Deng$^{1,*}$, Jan S. Hesthaven$^2$}
\cortext[cor2]{Corresponding author. E-mail: dengwh@lzu.edu.cn.}
\address{$^1$  School of Mathematics and Statistics, Gansu Key Laboratory of Applied Mathematics and Complex Systems, Lanzhou University, Lanzhou 730000, P.R. China}

\address{$^2$ EPFL-SB-MATHICSE-MCSS, \'{E}cole Polytechnique F\'{e}d\'{e}rale de Lausanne, CH-1015 Lausanne, Switzerland}


\begin{abstract}
This paper, as the sequel to  previous work, develops numerical schemes for fractional diffusion equations on a two-dimensional finite domain with triangular meshes.  We adopt the nodal discontinuous Galerkin methods for the full spatial discretization by the use of high-order nodal basis, employing multivariate Lagrange polynomials defined on the triangles. Stability analysis and error estimates are provided, which shows that if polynomials of degree $N$ are used, the methods are (N+1)-th order accurate for general triangulations. Finally, the performed numerical experiments confirm the optimal order of convergence.

\end{abstract}
\end{frontmatter}

\section{Introduction}
Historically, fractional calculus emerged nearly in the same time as classical calculus as a natural extension of classic calculus. However, its application for problems like fractional partial differential equations (FPDEs) are less mature than that associated with classic calculus. Only during the last few decade has  fractional calculus seen a broader application as a tool to describe a wide range of non-classical phenomena in the applied science and engineering, for example, the fractional Fokker-Planck equations for anomalous diffusion problems, continuous time random walk models with power law waiting time and/or jump length distribution (\cite{barkai, metzler}), and  the subdiffusion and superdiffusion process.

With the increasing utilization of fractional calculus it is necessary to develop appropriate and robust numerical methods to solve FPDEs for practical application. A fundamental difference between problems in classic calculus and fractional calculus lies in the non-local nature of the factional operators and the dependance of history, causing the essential difficulties and challenges for numerical approximation.

In recent years, however, successful work has emerged to deal with discretizing fractional models by adapting traditional numerical schemes, including finite difference methods, finite element methods, and spectral methods. Meerschaert and Tadjeran in [12] firstly proposed a stable difference method -- the shifted Gr\"{u}nwald-Letnikov formula -- to approximate fractional advection-dispersion flow equations. Recently, Tian et al. \cite{tian12}  put forward a higher order accurate numerical solution
method for the space fractional diffusion equation, referred to as the weighted and shifted Gr\"unwald difference operators. In \cite{xu1}, Li and Xu considered a space-time spectral method for solving the time fractional diffusion equation and Deng \cite{deng1} developed  finite element methods for discretizing the space and time fractional Fokker-Planck equation. More recently, Xu and Hesthaven \cite{xu14} discussed stable multi-domain spectral penalty methods for FPDEs and Zayernouri and Karniadakis \cite{Zayernouri} analyzed fractional Sturm-Liouville eigen-problems.

As an alternative, discontinuous Galerkin methods have also emerged. In 2010, Deng and Hesthaven \cite{deng2} proposed a local discontinuous Galerkin method for the fractional diffusion equation, and offered  stability analysis and error estimates, confirming that the schemes should exhibit  optional order of convergence for the superdiffusion case. Almost in the same time, Ji and Tang \cite{xia} presented  a purely qualitative study of the solution of spatial Caputo fractional problems in one and two dimensions using a high-order Runge-kutta discontinuous Galerkin methods, but did not offer theoretical results.

The main advantages of DG methods include geometric flexibility and the support of locally adapted resolution as well as excellent parallel efficiency. In \cite{xia}, the authors adopted the rectangular meshes to deal with the two-dimensional cases. This paper, as a successor to previous work \cite{deng2},  discusses how to approximate fractional diffusion equations with genuinely unstructured grids beyond one dimension and offer a theoretical analysis for the high dimensional case. To
focus on how to overcome the difficulties of developing LDG for fractional problems with unstructured meshes, we consider left \RL fractional equations and just the numerical experiments are provided for the right \RL fractional equations; in fact, the discussions for the right (or both left and right) \RL fractional equations are similar to the left ones.

The paper is organized as follows. In Section 2, we review the needed definitions of fractional operators and the fractional functional setting. In the following section, we propose our numerical schemes and show the detailed algorithm in computation. Section 4 gives the corresponding stability analysis and error estimates. The numerical results are provided in Section 5 and a few concluding remarks are given in the last section.

\section{Preliminaries}
In this section, we introduce some preliminary  definitions of fractional derivatives and associated functional setting for
the subsequent numerical schemes and theoretical analysis.

First we recall some definitions of the fractional derivatives and integrals listed as follows:
\begin{itemize}
  \item left \RL fractional derivative:
    $$ _aD_x^{\alpha}u(x) = \frac{1}{\Gamma(n-\alpha)} \frac{d^n}{dx^n} \int_{a}^{x} (x- \xi)^{n - \alpha -1}u(\xi) d\xi$$
  \item right \RL fractional derivative:
    $$ _xD_b^{\alpha}u(x) = \frac{(-1)^{n}}{\Gamma(n-\alpha)} \frac{d^n}{dx^n} \int_{x}^{b} (\xi - n)^{n - \alpha -1}u(\xi) d\xi$$
  \item left fractional integral:
    $$ _aD_x^{-\alpha}u(x) = \frac{1}{\Gamma(\alpha)}  \int_{a}^{x} {(x-\xi)^{\alpha-1}}{u(\xi)} d\xi$$
  \item right fractional integral:
      $$ _xD_b^{-\alpha}u(x) = \frac{1}{\Gamma(\alpha)}  \int_{x}^{b} {(\xi-x)^{\alpha-1}}{u(\xi)} d\xi,$$
\end{itemize}
where $\Gamma(\cdot)$ denotes the Gamma function, $\alpha \in [n-1,n)$, $x\in(a,b)$, a and b can be $-\infty$ and $+\infty$ respectively.

By simple linear transformations, the fractional integrals have the equivalent forms, i.e.,
\begin{equation}
\begin{split}
_aD_x^{-\alpha} u(x) = \frac{1}{\Gamma(\alpha)} \left( \frac{x-a}{2} \right)^{\alpha} \int_{-1}^{1} (1-\eta)^{\alpha-1}u\left(\frac{x+a}{2} + \frac{x-a}{2}\eta\right) d\eta,\\
_xD_b^{-\alpha} u(x) = \frac{1}{\Gamma(\alpha)} \left( \frac{b-x}{2} \right)^{\alpha} \int_{-1}^{1} (1+\eta)^{\alpha-1}u\left(\frac{b+x}{2} + \frac{b-x}{2}\eta\right) d\eta.
\end{split}
\label{eq2:2}
\end{equation}
On the basis of (\ref{eq2:2}), we use the Gauss-Jacobi quadrature with weight functions $(1-\eta)^{\alpha-1}$ and  $(1+\eta)^{\alpha-1}$ to solve the weakly singular fractional integrals in numerical computation. For the following Lemmas \ref{lemma2.1}, \ref{lem2.3}, Definitions \ref{defn2.2}, \ref{defn2.4}, and Theorem \ref{thm2.5}, we refer to \cite{Ervin:2006, deng1, deng2}.
\begin{lem}
\label{lemma2.1}
The left and right fractional integral operators are adjoint in the sense of the $L^2(a,b)$ inner product, i.e.,
\begin{equation}
(_aD_x^{-\alpha}u,v)_{L^2(a,b)} = (u,\, _{x}D_b^{-\alpha}v)_{L^2(a,b)}  \quad \forall \alpha>0, a<b.
\end{equation}
\end{lem}
In order to carry out the analysis, we need to introduce the fractional integral space here.
\begin{defn}\label{defn2.2}
Let $\alpha >0$. Define the norm
\begin{equation}
    \norm{u}_{H^{-\alpha}(\R)} :=  \norm{ |\omega|^{-\alpha} \widehat{u}}_{L^2(\R)},
\end{equation}
where $\widehat{u}(\omega)$ is the Fourier transform of $u(x)$. Let $H^{-\alpha}(\R)$ denote the closure of $\cinf{\R}$ with respect to $\norm{\cdot}_{H^{-\alpha}(\R)}$.
\end{defn}
%
%

\begin{lem} \label{lem2.3}
\begin{equation}
    (_{-\infty}D_x^{-\alpha}u,\smallskip _xD_{\infty}^{-\alpha} u) = cos(\alpha \pi)\norm{_{-\infty}D_x^{-\alpha}u}^2_{L^2(\R)} = cos(\alpha \pi) \norm{u}^2_{H^{-\alpha}(\R)}.
\end{equation}
\end{lem}
\noindent Let us now restrict attention to the case in which $ supp(v) \subset \Omega = (a,b)$. Then $_{-\infty}D_x^{-\alpha }u = {_a}D_x^{-\alpha}$,
and $_{x}D_{\infty}^{-\alpha} = {_x}D_b^{-\alpha}$. Straightforward extension of the above yields.
\begin{defn} \label{defn2.4}
Define the space $H_0^{-\alpha}(\Omega)$ as the closure of $\cinf{\Omega}$  with respect to $\norm{\cdot}_{H^{-\alpha}(\R)}$.
\end{defn}

The following theorem gives the inclusion relation between the fractional integral spaces with different $\alpha$.
\begin{thm}\label{thm2.5}
If $-\alpha_2 < -\alpha_1 < 0$, then $H_0^{-\alpha_1}(\Omega)$  is embedded into
$H_0^{-\alpha_2}(\Omega)$ and $L^2(\Omega)$ is embedded into both of them.
\end{thm}

\subsection{Notations for DG methods}
We consider problems posed on the physical domain $\Omega$ with boundary $\partial\Omega$ and assume that this domain is well approximated by
the computational domain $\Omega_h$.  Generally, we denote $\Omega_h$  as $\Omega$ when no misunderstanding is possible. This is a space filling triangulation composed of a collection of K geometry-conforming nonoverlapping elements, $D^k$, i.e.,
$$\Omega \simeq \Omega_h = \bigcup_{k=1}^{K}D^k,$$
and $\Gamma$ denotes the union of the boundaries of the elements $D^k$ of $\Omega_h$.
$\Gamma$ consists of two parts: the set of unique purely internal edges $\Gamma_i$ and the set of external edges $\Gamma_b=\partial \Omega$ of domain boundaries, and
$\Gamma = \Gamma_i\bigcup\Gamma_b$.
The shape of these elements can be arbitrary although we will mostly consider cases where they are d-dimensional  curvilinear simplices. In the two-dimensional case, the planar triangles are adopted.

Now we introduce the broken Sobolev space for any real number $s$,
$$
H^s(\Omega_h) = \{v\in L^2(\Omega): \forall k=1,2,\ldots, K, v|_{D^k} \in H^s(D^k)\},
$$
equipped with the broken Sobolev norm:
$$
\norm{v}_{H(\Omega_h)} = \big(\sum_{k=1}^{K}\norm{v}^2_{H^s(D^k)}\big)^{1/2}.
$$
Here, we retain the same style with the traditional Sobolev space when no misunderstanding is possible. When $s=0$, $H^0(\Omega_h) = L^2(\Omega_h).$

We introduce the approximation spaces $V_h$. Before that, we first define the space of N-th order polynomials in two variables on the 2-simplex $D$, $P_N^2(D)$ such that the dimension of the approximation polynomial space is
$$ \dim P_N^2(D) = N_p = \binom{N+2}{N},
$$
being the minimum space in which $P_N^2(D)$ may be complete. Let us introduce the nodal set $\{{\x}_i\}_{i=1}^{N_p}$, which are the nodal points or collocation points on $D$. These points must be chosen carefully to ensure well conditioned operators.  The polynomial space $P_N^2(D)$ can be illustrated as
\begin{eqnarray}
P_N^2(D) &=& \text{span} \{\ell_i(\x), i=1,2,\ldots, N_p \} \nonumber\\
         &=& \text{span} \{x^iy^j \ |\ (i,j)\geqslant0, i+j \leqslant N, (x,y) \in D^k \}, \nonumber
\end{eqnarray}
where $\ell_i(\x)$ denotes the two-dimensional multivariate Lagrange interpolation basis function. These two forms are equivalent since $P_N^2(D)$ is a finite dimensional polynomial space. For the explicit way to evaluate the genuinely two-dimensional Lagrange polynomials, we refer to \cite{jan1} and \cite{jan2}.

We also need to define a number of different inner products on the simplex, $D$. Consider the two continuous functions $f,g$, then the inner product, the associated $L^2$ norm and the inner product over the surface of D are defined as
$$
(f,g)_D = \int_D f(\x)g(\x) d\x, \quad (f,f) = ||f||_D^2, \quad (f,g)_{\partial D} = \int_{\partial D} f(\x) g(\x) d s .
$$
The corresponding global broken measures, inner products and norms are
$$
(f,g)_{\Omega} = \sum_{k=1}^{K} (f,g)_{D^k}, \quad (f,f) = \sum_{k=1}^{K} ||f||_{D^k} = ||f||_{\Omega}, \quad (f,g)_{\Gamma} = \sum_{k=1}^{K} (f,g)_{\partial D^k}.
$$

%


Next, we introduce some notations to manipulate numerical fluxes. For $e\in \Gamma$, we refer to the exterior information by a superscript `+' and to the interior information by a superscript `-'.
Using these notations, it is useful to define the average
$$\{u\} = \frac{u^++u^-}{2} \quad \text{on} \ e\in \Gamma_i,$$
$$ \{u\} = u \quad \text{on} \ e \in \Gamma_b, $$
where $u$ can be both a scalar and a vector.
In a similar fashion, we also define the jumps along a unit normal $\n$, as
$$[u]=\n^+u^+ + \n^-u^-, \ [\textbf{u}]=\n^+\cdot\textbf{u}^+ + \n^-\cdot\textbf{u}^- \quad \text{on} \ e\in \Gamma_i,$$
$$[u]=\n u, \ [\textbf{u}]=\n \cdot\textbf{u} \quad \text{on} \ e\in \Gamma_b.$$

Assume that the global solution can be approximated as
\begin{equation}
u(\x,t) \simeq u_h(\x,t) = \bigoplus_{k=1}^{K} u_h^{k}(\x,t) \in V_h = \bigoplus_{k=1}^{K} P_N^2(D^k),
\end{equation}
where $ P_N^2(D^k)$  is the space of N-th order polynomials defined on $D^k$.
The local solution, $u(\x,t)$ can be expressed by
\begin{equation}
u_h^k(\x,t)=  \sum_{i=1}^{N_p}u_h^k(\x_i,t) \ell_i^k(\x), \quad \x \in D^k,
\end{equation}
utilizing a nodal representation.

\section{The local nodal discontinuous Galerkin methods for fractional diffusion equations}
We consider two dimensional fractional problems
\begin{equation}
\frac{\partial u(\x,t)}{\partial t} = d_1 \frac{\partial^{\alpha}u(\x,t)}{\partial x^{\alpha}} + d_2 \frac{\partial^{\beta} u(\x,t)}{\partial y^{\beta}} + f(\x,t), \ \x = (x,y) \in \R^2,
\end{equation}
subject to appropriate boundary and initial conditions. Here, $\alpha, \beta \in (1,2]$, $d_1, d_2 > 0$, $f(\x,t)$ is a source term, and $\frac{\partial^{\alpha}}{\partial x^{\alpha}}, \frac{\partial^{\beta}}{\partial y^{\beta}}$ denote the left \RL fractional derivatives.
For  convenience and to enable the theoretical analysis, we restrict our problem to a homogeneous Dirichlet boundary condition on the form
\begin{equation}
 \left\{ \begin{array}{ll}
\frac{\partial u(\x,t)}{\partial t}
= d_1\frac{\partial}{\partial x}{_a}D_x^{\alpha-2}\frac{\partial}{\partial x}u(\x,t) \\
\qquad \qquad + d_2\frac{\partial}{\partial y}{_c}D_y^{\beta-2}\frac{\partial}{\partial y}u(\x,t) + f(\x,t) & \textrm{$(\x,t) \in \Omega \times[0,T] $},\\
u(\x,0) = u_0(\x) & \textrm{$\x \in \Omega$},\\
u(\x,t) = 0 & \textrm{$(\x,t) \in \partial \Omega\times [0,T]$,}
\end{array} \right.
\label{eq3:2}
\end{equation}
where $\Omega = (a,b)\times(c,d)$. For the simplicity of theoretical analysis, we set $d_1 = d_2 = 1$  without the loss of generality.

\subsection{The primal formulation}
Following the standard approach for the development of local discontinuous Galerkin methods for problems with higher order  derivatives, we introduce the auxiliary variables $\p = (p^x,p^y)$ and $\q= (q^x,q^y)$, and express the problem as
\begin{equation} \label{weakform}
\left\{ \begin{array}{ll}
\frac{\partial u(\x,t)}{\partial t} = \nabla \cdot \q + f(\x,t) &(\x,t)\in \Omega \times [0,T],\\
\q = (\frac{\partial^{\alpha-2}p^x}{\partial x^{\alpha-2}}, \frac{\partial^{\beta-2}p^y}{\partial y^{\beta-2}}) &(\x,t) \in \Omega\times[0,T],\\
\p = \nabla u     &(\x,t)\in  \Omega \times [0,T],\\
u(\x,0) = u_0(\x) &\x\in  \Omega ,\\
u(\x,t) = 0        &(\x,t)\in  \partial \Omega \times [0,T].\\
\end{array}\right.
\end{equation}
We set $h_k := \text{diam}(D^k)$ and $h := \text{max}_{k=1}^{K}h_k$.  Since the fractional integral spaces are embedded in $L^2(\Omega)$, we could assume that the exact solution $(u,\p,\q)$ of (\ref{weakform}) belongs to
\begin{equation}
H^1(0,T;H^1(\Omega_h))  \times (L^2(0,T;L^2(\Omega_h)))^2 \times (L^2(0,T;H^1(\Omega_h)))^2.
\end{equation}
Next, we require that $(u,\p,\q)$ satisfies the local formulation
\begin{eqnarray}
\left(\frac{\partial u(\x,t)}{\partial t},v\right)_{D^k} &=& \left(\n\cdot \q, v\right)_{\partial D^k} - (\q,\nabla v)_{D^k} + (f,v)_{D^k},\\
(\q, \fai)_{D^k} &=& \left(\left(\frac{\partial^{\alpha-2}p^x}{\partial x^{\alpha-2}},\frac{\partial^{\beta-2}p^y}{\partial y^{\beta-2}}\right), \fai\right)_{D^k},\\
(\p, \pai)_{D^k} &=& (u,\n\cdot\pai)_{\partial D^k} - (u, \nabla \cdot \pai)_{D^k},\\
(u(\cdot,0), v)_{D^k}  &=& (u_0(\cdot),v)_{D^k},
\end{eqnarray}
for all test functions $v \in H^1(\Omega_h)$, $\fai = (\phi^x, \phi^y) \in (L^2(\Omega_h))^2 = L^2(\Omega_h) \times L^2(\Omega_h)$, and $\pai =(\pi^x,\pi^y) \in (H^1(\Omega_h))^2 = H^1(\Omega_h) \times H^1(\Omega_h)$.

To complete  the primal formulation of our numerical schemes, we need to introduce the finite dimensional subspace of $H^1(\Omega_h)$, i.e.,
$$V_h = \{v:\Omega_h \rightarrow \mathbb{R}\big|\ v|_{D^k} \in P_N^2(D^k), \ k=1,2,\cdots, K \},$$
and then restrict the trial and test functions $v$ to $V_h$,
$\fai, \pai$ to $ (V_h)^2 = V_h \times V_h$ respectively. Furthermore we define $u_h, \p_h, \q_h$ as the approximation of $u,\p,\q$.
We then seek $(u_h,\p_h,\q_h) \in  H^1(0,T;V_h)\times (L^2(0,T;V_h))^2 \times (L^2(0,T;V_h))^2$ such that for all $v \in V_h, \ \fai, \pai \in (V_h)^2 $ the following holds:
\begin{eqnarray}
\left(\frac{\partial u_h(\x,t)}{\partial t},v\right)_{D^k} &=& (\n\cdot \flux{\q}_h, v)_{\partial D^k} - (\q_h,\nabla v)_{D^k} + (f,v)_{D^k}, \label{eq3:9}\\
(\q_h, \fai)_{D^k} &=& \left(\left(\frac{\partial^{\alpha-2}p_h^x}{\partial x^{\alpha-2}},\frac{\partial^{\beta-2}p_h^y}{\partial y^{\beta-2}}\right), \fai\right)_{D^k},\label{eq3:10}\\
(\p_h, \pai)_{D^k} &=& (\flux{u}_h,\n\cdot\pai)_{\partial D^k} - (u_h, \nabla \cdot \pai)_{D^k}. \label{eq3:11}
\end{eqnarray}
The form (\ref{eq3:9})-(\ref{eq3:11}),  obtained after integration by parts once, is known as the weak form, which is used for theoretical analysis in the following. For the computational part,
we introduce the strong form,  recovered by doing integration by parts once again partially, as
\begin{eqnarray}
\left(\frac{\partial u_h(\x,t)}{\partial t},v\right)_{D^k} &=& (\nabla \cdot \q_h, v)_{D^k} - (\n\cdot (\q_h - \flux{\q}_h), v)_{\partial D^k} ,\label{eq3:12} \\
(\q_h, \fai)_{D^k} &=& \left(\left(\frac{\partial^{\alpha-2}p_h^x}{\partial x^{\alpha-2}},\frac{\partial^{\beta-2}p_h^y}{\partial y^{\beta-2}}\right), \fai\right)_{D^k},\label{eq3:13}\\
(\p_h, \pai)_{D^k} &=& (\nabla u_h, \pai)_{D^k} - (u_h - \flux{u}_h,\n\cdot\pai)_{\partial D^k}. \label{eq3:14}
\end{eqnarray}
The two formulations are mathematically equivalent but computationally different \cite{jan2}. The boundary condition $u|_{\Gamma_b}=0$ will be imposed on the 3rd equation above. To guarantee  consistency, stability and optimal order of convergence of the formulation above, we must define the numerical flux $\flux{u}_h, \flux{\q}_h$ carefully.
In the two-dimensional case, we adopt the central flux, defined as
\begin{equation}
\flux{u}_h = \frac{u^+_h + u^-_h}{2}, \quad \flux{\q}_h =\frac{\q^+_h + \q^-_h}{2},
\end{equation}
at all internal edges, and at the external edges we use
\begin{equation}
\flux{u}_h = 0, \quad \flux{\q}_h = \q_h^+ = \q_h^-.
\end{equation}

\subsection{The semidiscrete scheme}
We will borrow the notations for various operators from \cite{jan1} to express the semidiscrete scheme.
Let us introduce some local and global vector and matrix notations to form the local statement 
\begin{displaymath}
\begin{array}{ll}
\mathbf{u}_h^k = [u_1^k, u_2^k, \cdots, u^k_{N_p}]^T; & \bm{u}_h = [\bm{u}_h^1; \bm{u}_h^2; \cdots, \bm{u}_h^K]^T, \\
(\mathbf{p}^k)^x_h = [(p^k)^x_1, (p^k)^x_2, \cdots, (p^k)^x_{N_p}]^T; & \p_h^x = [(\p^x)_h^1; (\p^x)_h^2; \cdots; (\p^x)_h^K]^T, \\
(\mathbf{p}^k)^y_h = [(p^k)^y_1, (p^k)^y_2, \cdots, (p^k)^y_{N_p}]^T; & \p_h^y = [(\p^y)_h^1; (\p^y)_h^2; \cdots; (\p^y)_h^K]^T, \\
(\mathbf{q}^k)^x_h = [(q^k)^x_1, (q^k)^x_2, \cdots, (q^k)^x_{N_p}]^T; & \q_h^x = [(\q^x)_h^1; (\q^x)_h^2; \cdots; (\q^x)_h^K]^T, \\
(\mathbf{q}^k)^y_h = [(q^k)^y_1, (q^k)^y_2, \cdots, (p^k)^y_{N_p}]^T; & \q_h^y = [(\q^y)_h^1; (\q^y)_h^2; \cdots; (\q^y)_h^K]^T.
\end{array}
\end{displaymath}
We also have the local mass matrix $M^k$ with
$$ M_{ij}^k = (\ell_i^k(\x),\ell_j^k(\x) )_{D^k}$$
and the local spatial stiff matrix $S_x^k, S_y^k$ with the entries
$$
(S_x^k)_{ij} = \left( \frac{\partial \ell_j (\x) }{\partial x}, \ell_i(\x) \right)_{D^k}, \quad
(S_y^k)_{ij} = \left( \frac{\partial \ell_j (\x) }{\partial y}, \ell_i(\x) \right)_{D^k}.
$$
It is a little complex to compute the fractional spatial stiff matrices of (\ref{eq3:13}) and they need to
be stored globally. However,  their elements depend on the affected regions only and they are sparse.

Next, we will discuss how to form  the global fractional spatial stiffness matrix in detail.
Denote $\widetilde{S}_x$ and $\widetilde{S}_y$ as the global fractional spatial stiffness matrices in the $x$ and $y$ direction, respectively.
$\widetilde{S}_x$ and $\widetilde{S}_y$ are $K \times K$ block matrices and every element is a $N_p \times N_p$ matrix.
Here, we use numerical quadrature on a triangle to compute these factional stiffness matrices.
We only describe the specific procedure of forming the global fractional spatial stiff matrix $\widetilde{S}_x$ for simplicity.
$\widetilde{S}_y$ is obtained in the same fashion.

The integral of a function $u$ defined on the physical element $D$ can be approximated by using the Gauss quadrature rule on triangle, i.e.,
$$ \int_{D} u(\x) d\x = J\int_{I} u(\x(\bm{r})) d\bm{r} \simeq J\sum_{j=1}^{Q_D}u(\x(\bm{r}_j))w_j = J\sum_{j=1}^{Q_D}u(\x _j )w_j,$$
where $Q_D$ is the total number of Gauss quadrature weights or points  and $J$ is the constant transformation Jacobian between the physical element and the reference element $I$.
The sets $\{ \bm{r}_j \}_{j=1}^{Q_D} \in I$ and $\{ \x_j \}_{j=1}^{Q_D} \in D$ are one-to-one correspondence by an affine map.
Thus, we have
\begin{equation}
\begin{split}
& (_aD_x^{\alpha-2}p_h^x,\ell^k(\x))_{D^k} \\ &\simeq  J^k \sum_{j=1}^{Q_D} {_a}D_x^{\alpha-2}p_h^x \ell^k(\x_j)w_j   \\
& = J^k \sum_{j=1}^{Q_D}\frac{1}{\Gamma(2-\alpha)}\bigg(\sum_{m\in A}\int_{x_{m-1}}^{x_m}(x_j-\xi)^{1-\alpha}
(p_h^x)^m(\xi,y_j)dx \\
&~~ ~ + \int_{x_{k-1}}^{x_j}(x_j-\xi)^{1-\alpha}(p_h^x)^k(\xi,y_j)dx\bigg) \ell^k(\x_j)w_j,
\end{split}
\end{equation}
where $A$ is the set of the indices of elements affected by each quadrature points on every triangle as sketched in Figure \ref{fig1}.
\begin{figure}
\centering
\includegraphics[width=0.5\textwidth]{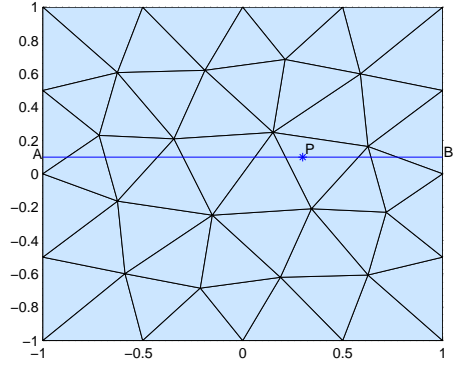}
\caption{All triangles (the ones intersected with the line $AB$) in $x$ direction affected by one Gauss quadrature point P on triangle $D^k$: the ones intersected with $AP$ is for the left fractional derivative/integral and the ones intersected with $PB$ is for the right fractional derivative/integral.}
\label{fig1}
\end{figure}

By using (\ref{eq2:2}) and the definition of the extension of the basis functions, we rewrite the above equation
\begin{equation}
\begin{split}
& (_aD_x^{\alpha-2}p_h^x,\ell^k(\x))_{D^k} \\
& \simeq \frac{J^k}{\Gamma(2-\alpha)} \sum_{j=1}^{Q_D} \bigg[\sum_{m\in A}\bigg(\left(\frac{x_j-x_{m-1}}{2}\right)^{2-\alpha}\int_{-1}^{1}(1-\eta)^{1-\alpha}\\
&~~~ \cdot(p_h^x)^m\left(\frac{x_j+x_{m-1}}{2} + \frac{x_j-x_{m-1}}{2}\eta,y_j\right)d\eta \\
&~~~ - \left(\frac{x_j-x_{m}}{2}\right)^{2-\alpha}\int_{-1}^{1}(1-\eta)^{1-\alpha} \\
&~~~ \cdot(p_h^x)^m\left(\frac{x_j+x_{m}}{2} + \frac{x_j-x_{m}}{2}\eta,y_j\right)d\eta \bigg)\\
&~~~ + \left(\frac{x_j-x_{k}}{2}\right)^{2-\alpha}\int_{-1}^{1}(1-\eta)^{1-\alpha} \\
&~~~ \cdot(p_h^x)^k\left(\frac{x_j+x_{k}}{2} + \frac{x_j-x_{k}}{2}\eta,y_j\right)d\eta \bigg]\ \ell^k(\x_j)w_j.
\end{split}
\end{equation}
Gauss quadratures with weight function $(1-\eta)^{1-\alpha}$ will be taken in the numerical process.
\begin{rem}
In the process of computation, we must pay attention to the fact that there is no link between the Lagrange interpolation points and Gauss quadrature points on the triangle.
\end{rem}
A piece of pseudocode illustrates the formation of the left global fractional spatial stiffness matrix.

\begin{algorithm}
\caption{Construction of the fractional global spatial stiffness matrix $_l\widetilde{S}_x$}
\begin{algorithmic}[1]
\State \% denote $\pmb{\ell}^k = [\ell_1^k, \ell_2^k, \cdots, \ell_{Np}^k]^T$
\State Initialize every block of $_l\widetilde{S}_x$ with zero matrix
\For{k = 1:K}
\For{j = 1:$Q_D$}
\State \% $Q_D$ is the total number of Gauss quadrature points
\State Find the set A and lXdir of every Gauss point $(x_j^k,y_j^k)$ on triangle $D^k$
\State \% lXdir is a length(A)$\times$2 matrix to store the intervals across by $PA$ on each triangle
\State $(_l\widetilde{S}_x) _{kk} =  (_l\widetilde{S}_x) _{kk} + (_{x_{k-1}}D_x^{\alpha-2}\pmb{\ell}^k(\x),\pmb{\ell}^k(\x))_{D^k}$
\For{m=1:length(A)}
\State $(_l\widetilde{S}_x) _{km} =  (_l\widetilde{S}_x) _{km} + (\frac{1}{\Gamma(2-\alpha)}\int_{x_{m-1}}^{x_m}(x-\xi)^{1-\alpha}\pmb{\ell}^m(\xi,y)d\xi,\pmb{\ell}^k(\x))_{D^k}$
\EndFor
\EndFor
\EndFor
\end{algorithmic}\label{alg1}
\end{algorithm}
\begin{rem}
It is unnecessary to store the full global fractional spatial stiffness matrices $_l\widetilde{S}_x$ and
$_l\widetilde{S}_y$ as they are both sparse.
\end{rem}

We recover the global semi-discrete form (\ref{eq3:12})-(\ref{eq3:14}) as
\begin{eqnarray}
M \frac{\partial \bm{u}_h}{\partial t} &=& S_x \q_h^x + S_y \q_h^y - \bigcup_{k=1}^K \int_{\partial D^k}\n \cdot ((\q_h^x, \q_h^y) - \flux{\q}_h) \ell(\x) d\bm{s},\\
M \q_h^x &=& _l\widetilde{S}_x \p_h^x, \quad M \q_h^y = {_l}\widetilde{S}_y \p_h^y, \\
M \p_h^x &=& S_x \bm{u}_h - \bigcup_{k=1}^K \int_{\partial D^k} n_x(\bm{u}_h - \flux{\bm{u}}_h) \ell(\x) d\bm{s}, \\
M \p_h^y &=& S_y \bm{u}_h - \bigcup_{k=1}^K \int_{\partial D^k} n_y(\bm{u}_h - \flux{\bm{u}}_h) \ell(\x) d\bm{s},
\end{eqnarray}
where $M, S_x, S_y$ are global mass and stiffness matrices and there non-zero diagonal block are constructed by $M^k, S_x^k, S_y^k$ respectively, and $\bigcup_{k=1}^K$
refers to the construction of the equations simultaneously for numerical computation.

\section{Stability analysis and error estimates}
Before initiating the theoretical analysis, we need following result:
\begin{lem}
Assume that $\Omega$ has been triangulated into $K$ elements, $D^k$. Then
\begin{equation}
\sum_{k=1}^{K} (\n\cdot\textbf{\emph{u}}, v)_{\partial D^k} = \oint_{\Gamma}\{\textbf{\emph{u}}\}\cdot [v] ds + \oint_{\Gamma_i}\{v\}[\textbf{\emph{u}}]ds.
\end{equation}
\end{lem}
The proof follows directly by rewriting and then summing the averages and jumps of all the terms along one edge.

By using this Lemma and summing all the terms of (\ref{eq3:9})-(\ref{eq3:11}), we  obtain the primal formulation:
Find $(u_h,\p_h,\q_h) \in H^1(0,T;V_h)\times (L^2(0,T;V_h))^2\times (L^2(0,T;V_h))^2$ such that
for all $(v,\fai,\pai) \in H^1(0,T;V_h)\times (L^2(0,T;V_h))^2\times (L^2(0,T;V_h))^2$ the following holds
\begin{equation}
B(u_h,\p_h,\q_h;v,\fai,\pai) = \mathcal{L}(v,\fai,\pai).
\end{equation}
We denote $(\cdot,\cdot) = (\cdot,\cdot)_{\Omega}$ when no misunderstanding is possible. Recall that the numerical flux is single valued and the homogeneous boundary condition is assumed. Then the discrete bilinear form $B$ can be defined as
\begin{equation}
\begin{split}
& B(u_h,\p_h,\q_h;v,\fai,\pai) \\ & :=  \int_{0}^{T}\left(\frac{\partial u_h}{\partial t},v\right) dt  + \int_{0}^{T}(\q_h,\nabla v)dt\\
& ~~~ + \int_{0}^{T} (\q_h,\fai)dt - \int_{0}^{T}\left(\left(\frac{\partial^{\alpha-2}p_h^x}{\partial x^{\alpha-2}},\frac{\partial^{\beta-2}p_h^y}{\partial y^{\beta-2}}\right),\fai \right)dt \\
&  ~~~  + \int_{0}^{T} (\p_h,\pai)dt + \int_{0}^{T}(u_h,\nabla\cdot\pai)dt\\
&  ~~~  - \int_{0}^{T} \left( \oint_{\Gamma}\flux{\q}_h\cdot[v]ds + \oint_{\Gamma_i}\flux{u}_h[\pai] ds \right)dt.
\end{split}
\label{eq4:3}
\end{equation}
The discrete linear form $\mathcal{L}$ is given by
\begin{equation}
\mathcal{L}(u_h,\p_h,\q_h;v,\fai,\pai) = \int_{0}^{T}(f,v)dt.
\end{equation}
Provided that the numerical fluxes $\flux{u}_h, \flux{\q}_h$ are consistent, the primal formulation (\ref{eq4:3}) is consistent, ensuring that the exact solution $(u,\p,\q)$ of (\ref{eq3:2}) satisfies
\begin{equation}
B(u,\p,\q;v,\fai,\pai) = \mathcal{L}(v,\fai,\pai).\label{eq4:5}
\end{equation}
for all $(v,\fai,\pai) \in H^1(0,T;V_h)\times (L^2(0,T;V_h))^2\times (L^2(0,T;V_h))^2$.

\subsection{Numerical stability}
Let $(\widetilde{u}_h,\widetilde{\p}_h,\widetilde{\q}_h) \in H^1(0,T;V_h)\times (L^2(0,T;V_h))^2\times (L^2(0,T;V_h))^2$ be the approximation of the solution
$(u_h,\p_h,\q_h)$. We denote $e_{u_h} := u_h - \widetilde{u}_h$, $e_{\p_h} := \p_h - \widetilde{\p_h}$ and $e_{\q_h} := \q_h - \widetilde{\q}_h$ as the roundoff errors.
\begin{thm}{($L^2$ stability).}
Numerical scheme (\ref{eq4:5}) with central flux is $L^2$ stable, and for all $t\in (0,T)$ its solution satisfies
\begin{equation}
\begin{split}
& \parallel e_{u_h}(\cdot,t)\parallel^2_{L^2(\Omega)} \\
 & =  \parallel e_{u_h}(\cdot,0)\parallel^2_{L^2(\Omega)} \\
& ~~~ - \ 2cos((\alpha/2-1)\pi)\int_{0}^{t} \int_{c}^{d}\parallel e_{p_h^x}(\cdot,y,t) \parallel^2_{H^{\frac{\alpha}{2}-1}} dydt\\
& ~~~ - \ 2cos((\beta/2-1)\pi) \int_{0}^{t} \int_{a}^{b}\parallel e_{p_h^y}(x,\cdot,t) \parallel^2_{H^{\frac{\beta}{2}-1}} dxdt.
\end{split}
\end{equation}
\label{thm1}
\end{thm}
\begin{proof}
We just prove the case $t = T$. From (\ref{eq4:3}), we recover the perturbation equation
\begin{equation}
B(e_{u_h},e_{\p_h},e_{\q_h};v,\fai,\pai) = 0
\end{equation}
for all $(v,\fai,\pai) \in H^1(0,T;V_h)\times (L^2(0,T;V_h))^2\times (L^2(0,T;V_h))^2$. Take $v = e_{u_h}$, $\fai = - e_{\p_h}$, and
$\pai = e_{\q_h}$, to obtain 
\begin{equation}
\begin{split}
0 &= B(e_{u_h},e_{\p_h},e_{\q_h};e_{u_h},-e_{\p_h},e_{\q_h}) \\
&= \frac{1}{2}\int_{0}^{T} \frac{\partial }{\partial t} \parallel e_{u_h}(\cdot,t)\parallel^2_{L^2(\Omega)} dt + \int_{0}^{T} \int_{\Omega}\nabla\cdot (e_{u_h}e_{\q_h})d\x dt \\
& ~~~ + \ 2cos((\alpha/2-1)\pi)\int_{0}^{T} \int_{c}^{d}\parallel e_{p_h^x}(\cdot,y,t) \parallel^2_{H^{\frac{\alpha}{2}-1}} dydt\\
& ~~~ + \ 2cos((\beta/2-1)\pi) \int_{0}^{T} \int_{a}^{b}\parallel e_{p_h^y}(x,\cdot,t) \parallel^2_{H^{\frac{\beta}{2}-1}} dxdt\\
& ~~~ - \int_{0}^{T} \left(\oint_{\Gamma}\flux{e}_{\q_h}\cdot[e_{u_h}]ds + \oint_{\Gamma_i}\flux{e}_{u_h}[e_{\q_h}]ds\right)dt.
\end{split}
\label{eq4:8}
\end{equation}
In (\ref{eq4:8}), integration by parts yields
\begin{equation}
\begin{split}
\int_{\Omega_h}\nabla\cdot(e_{u_h}e_{\q_h})d\x  & = \sum_{k=1}^{K}\int_{D^k}\nabla\cdot(e_{u_h}e_{\q_h})d\x
= \sum_{k=1}^{K}\oint_{\partial D^k}\n\cdot e_{\q_h}e_{u_h}ds\\
&= \oint_{\Gamma_i}(\n^+\cdot e_{\q_h}^+e_{u_h}^+ + \n^-\cdot e_{\q_h}^- e_{u_h}^-)ds + \oint_{\Gamma_b} \n \cdot e_{\q_h} e_{u_h} ds.
\end{split}\label{eq4:9}
\end{equation}
With the central flux we recover
\begin{equation}
\begin{split}
& \oint_{\Gamma}\flux{e}_{\q_h}\cdot[e_{u_h}]ds + \oint_{\Gamma_i}\flux{e}_{u_h}[e_{\q_h}]ds \\
&= \oint_{\Gamma_i}(\n^+\cdot e_{\q_h}^+e_{u_h}^+ + \n^-\cdot e_{\q_h}^- e_{u_h}^-)ds + \oint_{\Gamma_b} \n \cdot e_{\q_h} e_{u_h} ds.
\end{split}\label{eq4:10}
\end{equation}
Combining (\ref{eq4:8})-(\ref{eq4:10}), the desired result follows.
\end{proof}

\subsection{Error estimate}
For the error estimate, we define the orthogonal projection operators, $\mathbb{P}: H^1(\Omega_h)\rightarrow V_h$, $\mathbb{Q}:((L^2(\Omega_h))^2\rightarrow (V_h)^2$ and $\mathbb{S}: (H^1(\Omega_h))^2\rightarrow (V_h)^2$. For all the elements, $D^k,$ $k= 1,2,\cdots,K$,
$\mathbb{P}, \mathbb{Q}, \mathbb{S}$ are defined to satisfy
\begin{eqnarray}
(\mathbb{P}u-u, v)_{D^k} = 0 &\forall v\in P_N^2(D^k), \nonumber\\
(\mathbb{Q}\bm{u}-\bm{u}, \bm{v})_{D^k} = 0 &\forall \bm{v}\in (P_N^2(D^k))^2, \nonumber\\
(\mathbb{S}\bm{u}-\bm{u}, \bm{v})_{D^k} = 0 &\forall \bm{v}\in (P_N^2(D^k))^2. \nonumber
\end{eqnarray}

\begin{thm}
The $L^2$ error of the numerical scheme (\ref{eq4:5}) with a central flux satisfies
\begin{equation}
\parallel u(\x,t) - u_h(\x,t) \parallel_{L^2(\Omega_h)} \leqslant c(\alpha,\beta)h^{N+1}, \quad \alpha, \beta\in(1,2),
\end{equation}
where $c$ is dependent on $\alpha, \beta, \Omega$ and $N$ represents the order of polynomial.
\end{thm}

\begin{proof}
We denote
$$ e_u = u(\x,t) - u_h(\x,t), \ e_{\p}(\x,t) = \p(\x,t) - \p_h(\x,t), \ e_\q = \q(\x,t) - \q_h(\x,t);$$
and recover the error equation
\begin{equation}
B(e_u,e_\p,e_\q;v,\fai,\pai) = 0 \label{eq4:14}
\end{equation}
for all $(v,\fai,\pai) \in H^1(0,T;V_h)\times (L^2(0,T;V_h))^2\times (L^2(0,T;V_h))^2$.
Take $$ v = \mathbb{P}u - u_h,\ \fai = \p_h - \mathbb{Q}\p,\ \pai = \mathbb{S}\q - \q_h$$
in (\ref{eq4:14}). After rearranging terms, we obtain
\begin{equation}
B(v,-\fai,\pai;v,\fai,\pai) = B(v^e,-\fai^e,\pai^e;v,\fai,\pai), \label{eq4:15}
\end{equation}
where $v^e, \bm{w}^e$ and $\bm{z}^e$ are given as
$$v^e = \mathbb{P}u - u, \fai^e = \p - \mathbb{Q}\p, \pai^e = \mathbb{S}\q - \q.$$
Following the discussion in the proof of Theorem \ref{thm1}, the left side of (\ref{eq4:15}) becomes
\begin{equation}
\begin{split}
& B(v,-\fai,\pai;v,\fai,\pai) \\
& = \frac{1}{2}\int_{0}^T\frac{\partial}{\partial t}\parallel v(\cdot,t)\parallel^2_{L^2(\Omega_h)}dt \\
& ~~~ + cos((\alpha/2-1)\pi)\int_0^T\int_{c}^{d}\parallel \phi^x(\cdot,y,t)\parallel^2_{H^{\frac{\alpha}{2}-1}{(a,b)}}dydt\\
& ~~~ + cos((\beta/2-1)\pi)\int_0^T\int_{a}^{b}\parallel \phi^y(x,\cdot,t)\parallel^2_{H^{\frac{\beta}{2}-1}{(c,d)}}dxdt;
\end{split}
\end{equation}
and the right hand side can be expressed as
\begin{equation}
B(v^e,-\fai^e,\pai^e;v,\fai,\pai) = \Rmnum{1} + \Rmnum{2} + \Rmnum{3} + \Rmnum{4},
\end{equation}
where
\begin{equation}
\Rmnum{1} = \int_0^T \left(\frac{\partial v^e(\cdot,t)}{\partial t}, v(\cdot,t)\right)dt,
\end{equation}
\begin{equation}
\Rmnum{2} = \int_0^T\big[(\pai^e,\nabla v) + (v^e, \nabla\cdot\pai) + (\pai^e, \fai) -(\fai^e, \pai) \big]dt,
\end{equation}
\begin{equation}
\Rmnum{3} = -\int_{0}^{T} \left( \int_{\Gamma}\flux{\pai}^e\cdot[v] ds + \int_{\Gamma_i}v^e \ [\pai] ds \right) dt,
\end{equation}
\begin{equation}
\Rmnum{4} =  \int_0^T\left[   \left(\frac{\partial^{\alpha-2}{\phi^e}^x }{\partial x^{\alpha-2}},\phi^x\right)  + \left(\frac{\partial^{\beta-2}{\phi^e}^y }{\partial y^{\beta-2}},\phi^y\right) \right]dt.
\end{equation}
Using Cauchy-Schwarz inequality and  standard approximation theory, we obtain
\begin{equation}
\begin{split}
\Rmnum{1} &= \frac{1}{2}\int_0^T \parallel \frac{\partial v^e(\cdot,t)}{\partial t} \parallel^2_{L^2(\Omega_h)} dt + \frac{1}{2}\int_0^T \parallel v(\cdot,t)\parallel_{L^2(\Omega_h)} dt \\
&\leqslant ch^{2N+2} + \frac{1}{2}\int_0^T\parallel v(\cdot,t) \parallel^2_{L^2(\Omega)}dt,
\end{split}
\end{equation}
where c is a constant.

In $\Rmnum{2}$, all terms vanish because of the Galerkin orthogonality.
To deal with the term $\Rmnum{3}$, let us reexamine the construction of the Lagrange interpolation bases.
Since we require the same number of Lagrange interpolation points on every element, the basis functions of both elements are equal along the internal edge $e \in \Gamma_i$.
Applying the Cauchy-Schwarz inequality and the trace inequality, $\Rmnum{3}$ can be estimated as
\begin{eqnarray}
\Rmnum{3} &=& -\int_{0}^{T}  \int_{\Gamma_b}\flux{\pai}^e\cdot[v] ds dt \nonumber \\
          &\leqslant& \frac{1}{2} \int_{0}^{T} \int_{\partial \Omega}(\flux{\pai}^e)^2 ds
          + \frac{1}{2} \int_{0}^{T} \int_{\partial \Omega} ([v])^2ds dt \nonumber \\
          &\leqslant& c_1h^{2N+2} + c_2 \int_0^T\parallel v(\cdot,t)\parallel^2_{L^2(\Omega)}dt,
\end{eqnarray}
where $\Gamma_b = \partial\Omega$, $c_1$ and $c_2$ are related with $N$, $\Omega$. Due to the adjoint property and the embedding theorem for the fractional integral, by using the Young's inequality, $\Rmnum{4}$ can be expressed as
\begin{equation}
\begin{split}
\Rmnum{4} &\leqslant \int_0^T \bigg(  \frac{1}{2\varepsilon_1}\parallel {\phi^e}^x(\cdot,t) \parallel^2_{L^2(\Omega_h)} + \frac{\varepsilon_1}{2}\int_c^d\parallel\phi^x(\cdot,y,t)\parallel^2_{H^{\alpha-2}(a,b)}dy\\
&~~~ +\frac{1}{2\varepsilon_2}\parallel{\phi^e}^y(\cdot,t)\parallel^2_{L^2(\Omega_h)} + \frac{\varepsilon_2}{2} \int_a^b\parallel\phi^y(x,\cdot,t)\parallel^2_{H^{\beta-2}(c,d)}dx\bigg )dt\\
&\leqslant c/\varepsilon h^{2N+2} + c\varepsilon \int_0^T \Big( \int_c^d\parallel\phi^x(\cdot,y,t)\parallel^2_{H^{\frac{\alpha}{2}-1}(a,b)}dy\\
 &~~~ + \int_a^b\parallel\phi^y(x,\cdot,t)\parallel^2_{H^{\frac{\beta}{2}}(c,d)}dx \Big)dt,
\end{split}
\end{equation}
provided $\varepsilon$ is sufficiently small such that $cos((\alpha/2-1)\pi)>c\varepsilon$ and $cos((\beta/2-1)\pi)>c\varepsilon$.
 Combining all the above estimates, we  obtain
\begin{equation}
\begin{split}
& \frac{1}{2}\int_{0}^T\frac{\partial}{\partial t}\parallel v(\cdot,t)\parallel^2_{L^2(\Omega_h)}dt \\
& ~~~+ (cos((\alpha/2-1)\pi)-c\varepsilon)\int_0^T\int_{c}^{d}\parallel \phi^x(\cdot,y,t)\parallel^2_{H^{\frac{\alpha}{2}-1}{(a,b)}}dydt\\
& ~~~+ (cos((\beta/2-1)\pi)-c\varepsilon)\int_0^T\int_{a}^{b}\parallel \phi^y(x,\cdot,t)\parallel^2_{H^{\frac{\beta}{2}-1}{(c,d)}}dxdt \\
& \leqslant(c/\varepsilon)h^{2N+2} + c\int_0^T\parallel v(\cdot,t)\parallel^2_{L^2(\Omega_h)}dt. \\
\end{split}
\end{equation}
By simplifying further, this yields
\begin{equation}
\begin{split}
&\frac{1}{2}\parallel v(\cdot,T)\parallel^2_{L^2(\Omega_h)}\\
& ~~~+ (cos((\alpha/2-1)\pi)-c\varepsilon)\int_0^T \int_c^d\parallel\phi^x(\cdot,y,t)\parallel^2_{H^{\frac{\alpha}{2}-1}(a,b)}dydt  \\
& ~~~+ (cos((\beta/2-1)\pi)-c\varepsilon)\int_{0}^T\int_c^d\parallel\phi^y(x,\cdot,t)\parallel^2_{H^{\frac{\beta}{2}-1}(c,d)} dxdt \\
&\leqslant \frac{1}{2}\parallel v(\cdot,0)\parallel^2_{L^2(\Omega_h)} + (c/\varepsilon)h^{2N+2} + c\int_0^T\parallel v(\cdot,t)\parallel^2_{L^2(\Omega_h)}dt \\
&\leqslant (c/\varepsilon)h^{2N+2} + c\int_0^T\parallel v(\cdot,t)\parallel^2_{L^2(\Omega_h)}dt.
\end{split}
\end{equation}
From the Gr\"{o}nwall's lemma and  standard approximation theory, the desired result follows.
\end{proof}

\section{Numerical results}
In this section, we offer a few numerical results to validate analysis. To deal with the method-of-line fractional PDE, i.e., the classical ODE system, we utilize a low-storage five stage fourth order explicit Runge-Kutta method \cite{jan1}. To ensure the overall error is dominated by space error, small time steps are used.

\noindent \textbf{Example 1} We consider the problem
\begin{equation}
\frac{\partial u(x,y,t)}{\partial t} = {_{-1}}D_x^{\alpha}u(x,y,t) + {_{-1}}D_y^{\alpha}u(x,y,t)+ f(x,y,t),
\end{equation}
where
\begin{displaymath}
\begin{split}
f(x,y,t)
&= -e^{-t}\big((x^2-1)^3(y^2-1)^3 + (y^2-1)^3_{-1}D_x^{\alpha-2}(6(x^2-1)(5x^2-1))\\
&+ (x^2-1)^3_{-1}D_y^{\beta-2}(6(y^2-1)(5y^2-1))\big)
\end{split}
\end{displaymath}
on the computational domain $\Omega=(-1,1)\times(-1,1)$ and $\alpha, \beta \in (1,2)$.
We consider the initial condition
\begin{equation}
u(x,y,0) = (x^2-1)^3(y^2-1)^3
\end{equation}
and the Dirichlet boundary condition
\begin{equation}
u(x,y,t) = 0, \quad (x,y) \in \partial \Omega.
\end{equation}
The exact solution is $u(x,y,t)=e^{-t}(x^2-1)^3(y^2-1)^3$.


\begin{table}
\centering
\caption{Numerical errors ($L_2$) and order of convergence on unstructured meshes for Example 1. N denotes the order of polynomial in two variables, and K is the total number of triangle elements.}
    \begin{tabular}{*{9}{|c}|}
      \hline
      &K & \multicolumn{1}{c|}{32} & \multicolumn{2}{c|}{137} & \multicolumn{2}{c|}{378} & \multicolumn{2}{c|}{899} \\
    \hline
     N &$(\alpha,\beta)$ &  error  &  error & order &  error & order &  error & order   \\ \hline
     \multirow{4}{*}{1} & (1.01,1.01)  &1.55e-1 &3.48e-2 &2.16 &1.27e-2 &1.99 & 5.39e-3 &2.11 \\
                    \cline{2-9}
                    & (1.5,1.5) & 1.44e-1 &3.21e-2 &2.17 &1.22e-2 &1.94 &5.09e-3 &2.16 \\
                    \cline{2-9}
                    & (1.8,1.8) & 1.34e-1 &3.03e-2 &2.14 &1.13e-2 &1.93 &5.40e-3 &1.82\\
                    \cline{2-9}
                    &(1.99,1.99)& 1.27e-1 &2.94e-2 &2.11 &1.07e-2 &1.98 &4.48e-3 &2.15\\
                    \hline
      &K & \multicolumn{1}{c|}{32} & \multicolumn{2}{c|}{137} & \multicolumn{2}{c|}{378} & \multicolumn{2}{c|}{899} \\\hline
     \multirow{4}{*}{2} & (1.01,1.01) &1.65e-2 &4.39e-3 &2.62 &5.22e-4 &2.92 &1.44e-4 &2.97 \\
                    \cline{2-9}
                    & (1.5,1.5) &1.40e-2 &3.47e-3 &2.76 &3.41e-3 &3.18 &9.19e-5 &3.02 \\
                    \cline{2-9}
                    & (1.8,1.8) &1.33e-2 &1.57e-3 &3.08 &3.37e-4 &3.01 &9.20e-5 &3.20 \\
                    \cline{2-9}
                    &(1.99,1.99)&1.36e-2 &1.57e-3 &3.11 &3.56e-4 &2.90 &9.75e-5 &3.20\\
                    \hline
      &K & \multicolumn{1}{c|}{88} & \multicolumn{2}{c|}{137} & \multicolumn{2}{c|}{378} & \multicolumn{2}{c|}{562} \\  \hline
     \multirow{4}{*}{3} & (1.01,1.01) &5.51e-4 &2.17e-4  &4.20 &3.07e-5 &3.86 &1.53e-5 &3.82\\
                    \cline{2-9}
                    & (1.5,1.5) &3.73e-4  &1.38e-4 &4.49 &1.88e-5 &3.93 &8.53e-6 &4.33\\
                    \cline{2-9}
                    & (1.8,1.8) &3.30e-4  &1.33e-4 &4.10 &2.00e-5 &3.73 &8.05e-6 &4.99 \\
                    \cline{2-9}
                    &(1.99,1.99)&3.83e-4  &1.61e-4 &3.92 &2.55e-5 &3.63 &1.81e-5 &4.23\\
                    \hline
    \end{tabular}

    \label{table1}
\end{table}

\begin{figure}
\centering
\includegraphics[width=0.35\textwidth]{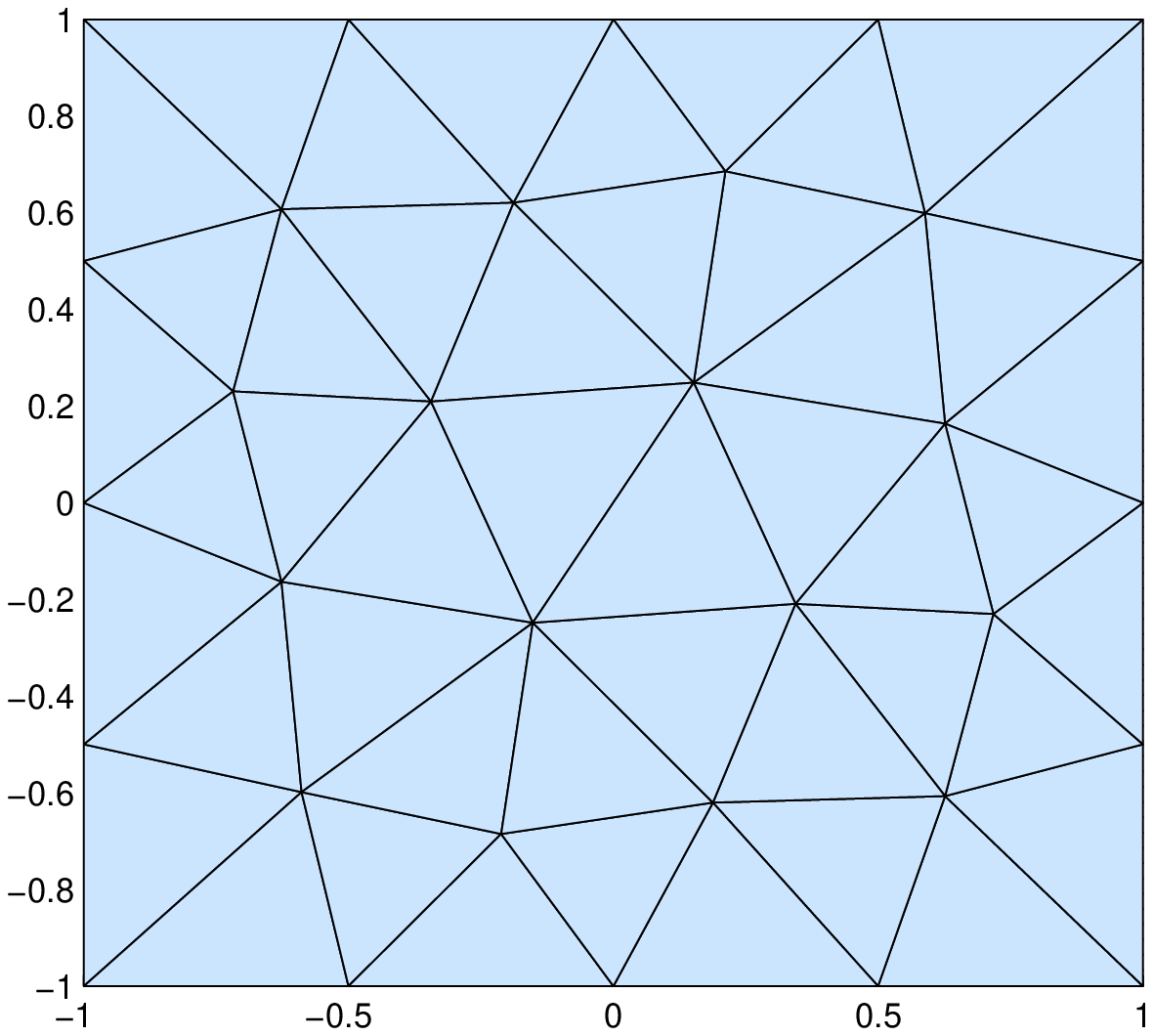}
\includegraphics[width=0.37\textwidth]{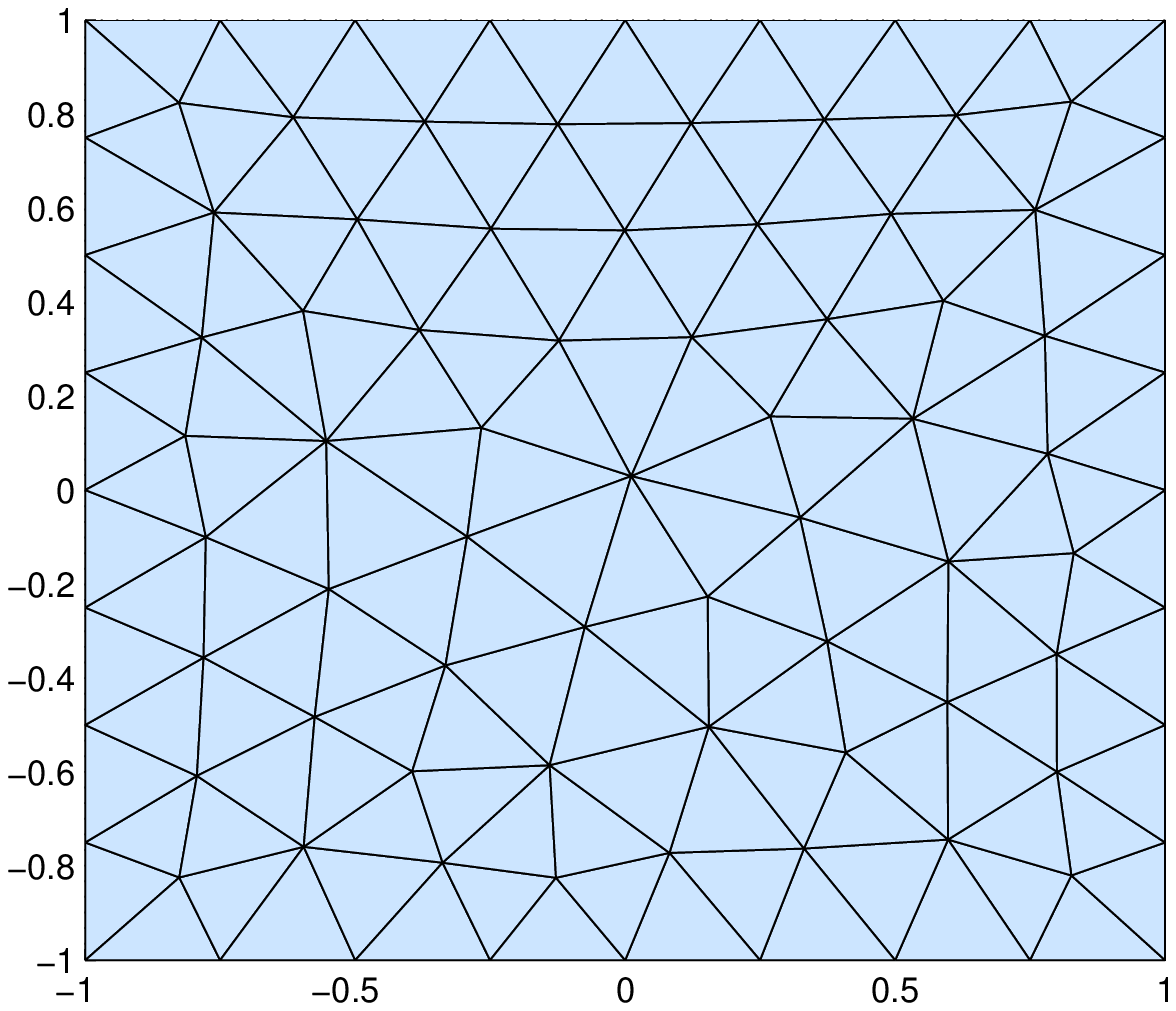}\\
\includegraphics[width=0.35\textwidth]{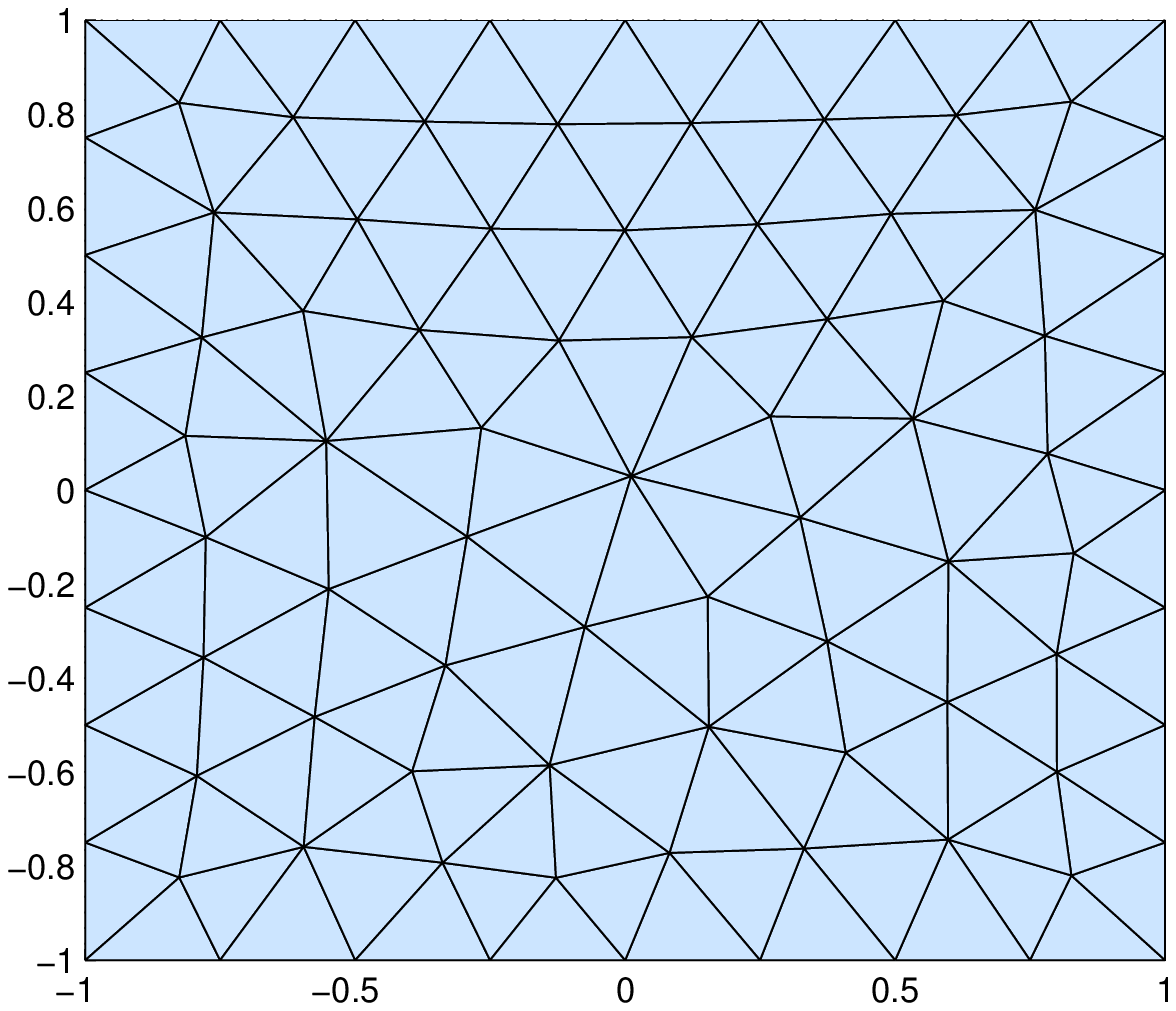}
\includegraphics[width=0.38\textwidth]{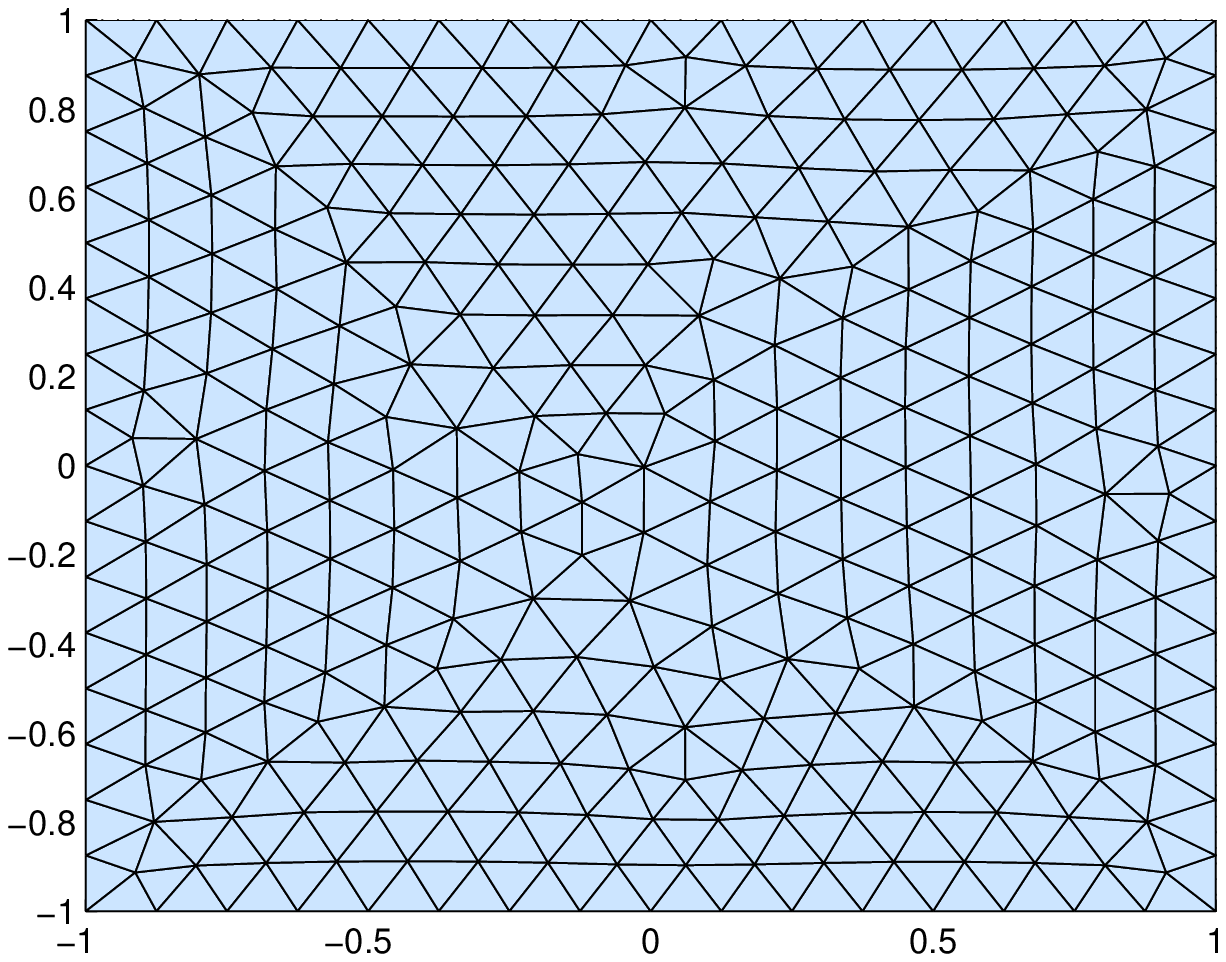}
\caption{Partial unconstructed meshes used in numerical examples }
\label{fig5.2}
\end{figure}

\noindent \textbf{Example 2} Now we expand our algorithm to fractional diffusion equations with both the left and right fractional derivatives.
We consider the model
\begin{equation}
\begin{split}
\frac{\partial u}{\partial t}& = d_+  {_{-1}}D_x^{\alpha}u(x,y,t) +  d_-  {_x}D_1^{\alpha}u(x,y,t) \\
& ~~~ + e_+  {_{-1}}D_y^{\beta}u(x,y,t) +  e_-  {_y}D_1^{\beta}u(x,y,t)+ f(x,y,t),
\end{split}
\end{equation}
with the initial condition
$$
u(x,y,0) = (x^2-1)^3(y^2-1)^3
$$
and boundary condition
$$
u(x,y,t)|_{\partial \Omega} = 0.
$$
Here, $(x,y)\in (-1,1) \times (-1,1)$, $\alpha, \beta \in(1,2)$ and $d_+=d_1=1$, $e_+=e_1=1$,
\begin{equation}
\begin{split}
f(x,y,t) &= -e^{-t}\bigg( (x^2-1)^3(y^2-1)^3 \\
&+ (_{-1}D_x^{\alpha-2}+{_x}D_1^{\alpha-2})\big(6(x^2-1)(5x^2-1)\big) (y^2-1)^3 \\
&+(_{-1}D_y^{\beta-2}+{_y}D_1^{\beta-2})\big(6(y^2-1)(5y^2-1)\big) (x^2-1)^3\bigg).
\end{split}
\nonumber
\end{equation}
The exact solution is $u(x,y,t) = e^{-t}(x^2-1)^3(y^2-1)^3$.

\begin{table}
\centering
\caption{Numerical errors ($L_2$) and order of convergence on unstructured meshes for Example 2. N denotes the order of polynomial in two variables, and K is the total number of triangle elements.}
    \begin{tabular}{*{9}{|c}|}
      \hline
      &K & \multicolumn{1}{c|}{32} & \multicolumn{2}{c|}{137} & \multicolumn{2}{c|}{378} & \multicolumn{2}{c|}{899} \\
    \hline
     N &$(\alpha,\beta)$ &  error  &  error & order &  error & order &  error & order   \\ \hline
     \multirow{4}{*}{1} & (1.01,1.01)  &1.54e-1 &3.45e-2 &2.16 &1.24e-2 &2.00 & 5.19e-3 &2.15 \\
                    \cline{2-9}
                    & (1.1,1.8) & 1.32e-1 &3.11e-2 &2.08 &1.18e-2 &1.91 &5.21e-3 &2.01 \\
                    \cline{2-9}
                    & (1.5,1.5) & 1.24e-1 &2.97e-2 &2.17 &1.09e-2 &1.96 &5.19e-3 &2.16 \\
                    \cline{2-9}
                    &(1.99,1.99)& 1.13e-1 &2.66e-2 &2.09 &9.69e-3 &1.98 &4.05e-3 &2.15\\
                    \hline
      &K & \multicolumn{1}{c|}{32} & \multicolumn{2}{c|}{137} & \multicolumn{2}{c|}{378} & \multicolumn{2}{c|}{899} \\\hline
     \multirow{4}{*}{2} & (1.01,1.01) &1.64e-2 &2.25e-3 &2.87 &5.04e-4 &2.93 &1.66e-4 &2.74 \\
                    \cline{2-9}
                    & (1.5,1.5) &1.32e-2 &1.62e-3 &3.03 &3.29e-3 &3.12 &9.96e-5 &2.94 \\
                    \cline{2-9}
                    & (1.8,1.1) &1.48e-2 &2.18e-3 &2.76 &4.69e-4 &3.00 &1.28e-4 &3.21 \\
                    \cline{2-9}
                    &(1.99,1.99)&1.44e-2 &1.64e-3 &3.13 &3.88e-4 &2.82 &1.10e-4 &3.10\\
                    \hline
      &K & \multicolumn{1}{c|}{32} & \multicolumn{2}{c|}{88} & \multicolumn{2}{c|}{137} & \multicolumn{2}{c|}{378} \\  \hline
     \multirow{4}{*}{3} & (1.01,1.01) &3.42e-3 &5.23e-4  &3.67 &2.01e-4 &5.25 &2.47e-5 &4.10\\
                    \cline{2-9}
                    & (1.5,1.5) &2.36e-3  &3.01e-4 &4.031 &1.10e-4 &5.53 &1.39e-5 &4.05\\
                    \cline{2-9}
                    & (1.2,1.6) &2.62e-3  &3.62e-4 &3.88 &1.46e-4 &4.98 &2.31-5 &3.61 \\
                    \cline{2-9}
                    &(1.99,1.99)&2.36e-4  &3.01e-4 &4.03 &1.10e-5 &5.52 &1.39e-5 &4.05\\
                    \hline
    \end{tabular}

    \label{table2}
\end{table}

In Table \ref{table1} and Table \ref{table2}, we note a convergence rate $\mathcal{O}(h^{N+1})$, which demonstrates that the central flux is optimal in two dimension.
Note that the small deviations from the optimal order are caused by the unstructured nature of the grid and the numerical quadrature for fractional integral operators.

\section{Conclusions}
In this paper, we propose  nodal discontinuous Galerkin methods for two dimensional \RL fractional equations on unstructured meshes, and offer a stability analysis and error estimates. Numerical experiments confirm that the optimal order of convergence is recovered. For the convenience of the theoretical analysis we have restricted our models to homogeneous boundary conditions, although this is not an essential restriction.

\section*{Acknowledgements}
This work was supported by NSFC 11271173, NSF DMS-1115416, and OSD/AFOSR FA9550-09-1-0613.

\appendix
\begin{center}
  {\bf APPENDIX}
\end{center}
For the right fractional integral, similarly we also have
\begin{equation}
\begin{split}
( _xD_b^{\alpha-2} p_h^x,\ell^k(\x))_{D^k} &\simeq  J^k \sum_{j=1}^{Q_D} {_x}D_b^{\alpha-2}p_h^x \ell^k(\x_j)w_j   \\
& = J^k \sum_{j=1}^{Q_D}\frac{1}{\Gamma(2-\alpha)}\bigg(\sum_{t\in B}\int_{x_{t-1}}^{x_t}(\xi - x_j)^{1-\alpha}
(p_h^x)^t(\xi,y_j)dx \\
&~~~ + \int_{x_{j}}^{x_{k-1}}(\xi - x_j)^{1-\alpha}(p_h^x)^k(\xi,y_j)dx\bigg) \ell^k(\x_j)w_j,
\end{split}
\end{equation}
where B is like A we introduce before, i.e., the index of triangles crossed by $PB$ in Figure \ref{fig1}. Then
\begin{equation}
\begin{split}
(_xD_b^{\alpha-2} p_h^x,\ell^k(\x))_{D^k}
& \simeq \frac{J^k}{\Gamma(2-\alpha)} \sum_{j=1}^{Q_D} \bigg[\sum_{m\in B}\bigg(\left(\frac{x_{m} - x_j}{2}\right)^{2-\alpha}\int_{-1}^{1}(1+\eta)^{1-\alpha}\\
&~~~ \cdot(p_h^x)^m\left(\frac{x_m+x_j}{2} + \frac{x_m-x_j}{2}\eta,y_j\right)d\eta \\
&~~~ - \left(\frac{x_m-x_{j-1}}{2}\right)^{2-\alpha}\int_{-1}^{1}(1+\eta)^{1-\alpha} \\
&~~~\cdot(p_h^x)^m\left(\frac{x_{m}+x_{j-1}}{2} + \frac{x_{m}-x_{j-1}}{2}\eta,y_j\right)d\eta \bigg)\\
&~~~ + \left(\frac{x_k-x_{j}}{2}\right)^{2-\alpha}\int_{-1}^{1}(1+\eta)^{1-\alpha} \\
&~~~ \cdot(p_h^x)^k\left(\frac{x_k+x_{j}}{2} + \frac{x_k-x_{j}}{2}\eta,y_j\right)d\eta ) \bigg]\ \ell^k(\x_j)w_j.
\end{split}
\end{equation}
Now the following is a piece of pseudocode for the right fractional stiffness matrix. 
\begin{algorithm}
\caption{Construction of the fractional global spatial stiffness matrix $_r\widetilde{S}_x$}
\begin{algorithmic}[1]
\State \% denote $\pmb{\ell}^k = [\ell_1^k, \ell_2^k, \cdots, \ell_{Np}^k]^T$
\State Initialize every block of $_r\widetilde{S}_x$ with zero matrix
\For{k = 1:K}
\For{j = 1:$Q_D$}
\State \% $Q_D$ is the total number of Gauss quadrature points
\State Find the set B and rXdir of every Gauss point $(x_j^k,y_j^k)$ on triangle $D^k$
\State \% rXdir is a length(B)$\times$2 matrix to store the intervals across by $PB$ on each triangle
\State $(_r\widetilde{S}_x) _{kk} =  (_r\widetilde{S}_x) _{kk} + (_xD_{x_{k-1}}^{\alpha-2}\pmb{\ell}^k(\x),\pmb{\ell}^k(\x))_{D^k}$
\For{t=1:length(B)}
\State $(_r\widetilde{S}_x) _{kt} =  (_r\widetilde{S}_x) _{kt} + (\frac{1}{\Gamma(2-\alpha)}\int_{x_{t-1}}^{x_t}(\xi-x)^{1-\alpha}\pmb{\ell}^t(\xi,y)d\xi,\pmb{\ell}^k(\x))_{D^k}$
\EndFor
\EndFor
\EndFor
\end{algorithmic}\label{alg2}
\end{algorithm}

\end{document}